\documentclass[12pt,reqno,fleqn]{amsart}  
\usepackage{tikz}
\usepackage{amsmath,amstext,amsthm,amssymb,amsxtra}
\usepackage{mathrsfs}
\usepackage[top=1.5in, bottom=1.5in, left=1.25in, right=1.25in]	{geometry}
\usepackage{lmodern}

\usepackage{graphics}
\usepackage{euscript}
\usepackage{xcolor}

\usepackage{mathtools}
\mathtoolsset{showonlyrefs,showmanualtags}

\usepackage{hyperref} 
\hypersetup{
	colorlinks=true,       
	linkcolor=blue,          
	citecolor=magenta,        
	filecolor=magenta,      
	urlcolor=cyan           
}

\usepackage[msc-links]{amsrefs}

\newtheorem{theorem}{Theorem}[section]
\newtheorem{problem}{Problem}

\newtheorem{lemma}{Lemma}[section]
\newtheorem{corollary}{Corollary}[section]
\newtheorem{OldTheorem}{Theorem}
\theoremstyle{definition}

\theoremstyle{definition}

\newtheorem{remark}{Remark}[section]

\theoremstyle{remark}

\def\sign{{\rm sign\,}}

\def\ZF{\ensuremath{\mathscr  F}}

\def\constant{{\rm constant}}

\def\ZU{\ensuremath{\mathbb U}}

\def\ZS{\ensuremath{\mathfrak S}}
\def\MM^d{\ensuremath{\mathfrak M}}
\def\MM{\ensuremath{\mathcal M}}

\def\ZB{\ensuremath{\mathscr B}}

\def\ZI{\ensuremath{\textbf 1}}

\def\ZN{\ensuremath{\mathbb N}}

\def\ZD{\ensuremath{\mathscr  D}}

\def\ZR{\ensuremath{\mathbb R}}

\def\pr{\ensuremath{\mathrm {pr}}}

\def\ZE{\ensuremath{\mathbf E}}

\numberwithin{equation}{section}
\def\md#1#2\emd{\ifx0#1
	\begin{equation*} #2 \end{equation*}\fi  
	\ifx1#1\begin{equation}#2\end{equation}\fi   
	\ifx2#1\begin{align*}#2\end{align*}\fi   
	\ifx3#1\begin{align}#2\end{align}\fi    
	\ifx4#1\begin{gather*}#2\end{gather*}\fi  
	\ifx5#1\begin{gather}#2\end{gather}\fi   
	\ifx6#1\begin{multline*}#2\end{multline*}\fi  
	\ifx7#1\begin{multline}#2\end{multline}\fi  
	\ifx8#1\begin{multline*}\begin{split}#2\end{split}\end{multline*}\fi
	\ifx9#1\begin{multline}\begin{split}#2\end{split}\end{multline}\fi
}
\newcommand {\e }[1]{\eqref{#1}}
\newcommand {\lem }[1]{Lemma \ref{#1}}

\newcommand {\cor }[1]{Corollary \ref{#1}}

\newcommand {\trm }[1]{Theorem \ref{#1}}

\title[] {On the best constants in Khintchine type \\inequalities for martingales}
\author{Grigori A. Karagulyan}
\address{Faculty of Mathematics and Mechanics, Yerevan State
	University, Alex Manoogian, 1, 0025, Yerevan, Armenia} 
\email{g.karagulyan@ysu.am}
\address{Institute of Mathematics NAS RA, Marshal Baghramian ave., 24/5, Yerevan, 0019, Armenia} 
\email{g.karagulyan@gmail.com}
\thanks{The work was supported by the Higher Education and Science Committee of RA, in the frames of the research project 21AG‐1A045 }
\begin{document}
	\begin{abstract}
		For discrete martingale-difference sequences $d=\{d_1,\ldots,d_n\}$ we consider Khintchine type inequalities, involving certain square function $\ZS(d)$ considered by Chang-Wilson-Wolff in \cite{CWW}. In particular, we prove
		\begin{equation}\label{x72}
			\left\|\sum_{k=1}^nd_k\right\|_p\le 2^{1/2}\big(\Gamma((p+1)/2))/\sqrt{\pi}\big)^{1/p}\|\ZS(d)\|_\infty,\quad p\ge 3,
		\end{equation}
		where the constant on the right hand side is the best possible and the same as known for the Rademacher sums $\sum_{k=1}^na_kr_k$. Moreover, for a fixed $n$ the constant in \e{x72} can be replaced by $\sum_{k=1}^nr_k/\sqrt{n}$. We apply a technique, reducing the general case to the case of Haar and Rademacher sums, that allows also establish a sub-Gaussian estimate
		\begin{equation}
			\ZE\left[\exp\left(\lambda\cdot \left(\frac{\sum_{k=1}^nd_k}{\|\ZS(d)\|_\infty}\right)^2\right)\right]\le \frac{1}{\sqrt{1-2\lambda}},\quad 0<\lambda<1/2,
		\end{equation} 
	where the constant on the right hand side is the best possible.
	\end{abstract}
	\maketitle  
	\section{Introduction}

\subsection{Two martingale square functions}
	Let $(X,\ZF, \mu)$ be a probability space. We consider a discrete filtration, i.e. a sequence of $\sigma$-algebras $\ZF_n\subset \ZF_{n+1}\subset \ZF$, $n=0,1,2,\ldots$, such that each $\ZF_n$ is generated by a finite or countable family $\ZD_n$ of disjoint measurable sets called a partition. Namely, we suppose $\ZD_0=\{X\}$ and each set $A\in \ZD_n$, $n\ge 0$, is a union of some elements of $\ZD_{n+1}$, which we call children of $A$. Given $x\in X$ denote by $V_n(x)$ the unique element of $\ZD_n$ containing the point $x$. In this setting the probability conditional expectation of a random variable $f$ with respect to the $\sigma$-algebra $\ZF_n$ may be written by
	\begin{equation}
		\ZE(f|\ZF_n)(x)=\frac{1}{\mu(V_n(x))}\int_{V_n(x)}f,\quad x\in X.
	\end{equation}
	Recall that a finite sequence of random variables $f=\{f_0, f_1,\ldots,f_n\}$ is said to be a martingale with respect to a filtration $\{\ZF_k\}$ (or a sequence of partitions $\{\ZD_k\}$), if $\ZE(f_k|\ZF_{k-1})=f_{k-1}$, $k=1,2,\ldots,n$. Let $d_k=f_k-f_{k-1}$, $k=1,2,\ldots,n$  be the difference sequence of $f$. 
	Then the definition of the martingale may be equivalently written by $\ZE(d_k|\ZF_{k-1})=0$, $k\ge1$, which in the case of discrete filtration means
	\begin{equation}\label{x57}
		\int_Vd_k=0 \hbox{ whenever }V\in \ZD_{k-1},\quad k=1,2,\ldots,n.
	\end{equation} 
Consider the following non-classical and classical martingale square functions
\begin{align}
	&\ZS f(x)=\left(\sum_{k=1}^n |D_k(x)|^2\right)^{1/2},\hbox{ where } D_k(x)=\sum_{V\in \ZD_{k-1}}\|\ZI_{V}d_k\|_\infty\cdot \ZI_V(x),\label{x42}\\
	&Sf(x)=\left(\sum_{k=1}^n |d_k(x)|^2\right)^{1/2}.\label{x53}
\end{align}
Clearly, we have $S(f)\le \ZS(f)$, since it is easy to check that $|d_k(x)|\le D_k(x)$. In general, these square functions may differ depending on the martingale. We have an equality $S(f)= \ZS(f)$ whenever our filtration $\{\ZF_k\}$ is dyadic, i.e. any element $V\in \ZD_{k-1}$ has exactly two children $V^+,V^-\in \ZD_k$ with $\mu(V^+)=\mu(V^-)=\mu(V)/2$. In the dyadic case we will simply have $D_k(x)=|d_k(x)|$. If the elements of $\ZD_k$ are dyadic intervals of $[0,1)$, then we get a Haar martingale, i.e. we can write 
\begin{equation*}
	f_k=\sum_{j=1}^{2^{k+1}}a_jh_j,\quad k=1,2,\ldots,n,
\end{equation*} 
where $h_j$, $j=1,2,\ldots$, are the Haar functions, defined by 
\begin{align}
	&h_1(x)\equiv 1,\quad h_n(x)=\ZI_{\Delta_n^+}(x)-\ZI_{\Delta_n^-}(x),\quad n\ge 2,\\
	&\Delta_n^+=\left[\frac{j-1}{2^k},\frac{2j-1}{2^{k+1}}\right),\quad \Delta_n^-=\left[\frac{2j-1}{2^{k+1}},\frac{j}{2^k}\right)\label{x66}
\end{align}
if $n=2^k+j$, $n\ge0$, $1\le j\le 2^k$. In the case of Rademacher martingale, i.e. if $d_k(x)=a_kr_k(x)$, where $r_k(x)=\sin(2^k\pi x)$ are the Rademacher independent random variables, we have $D_k(x)=|d_k(x)|\equiv |a_k|$. Thus both square functions are constant and we can write
\begin{equation}
	\ZS(f)=S(f)=\left(\sum_{k=1}^na_k^2\right)^{1/2}.
\end{equation}

\subsection{Sub-Gaussian estimates}
Square function $\ZS$ was first considered by Chang-Wilson-Wolff in \cite{CWW}, proving a sub-Gaussian estimate for infinite sequences of martingales.  The result of \cite{CWW} can be equivalently stated for finite martingale sequences as follows.
\begin{OldTheorem}[\cite{CWW}]\label{OT1}
	If $f=\{f_0,f_1,\ldots,f_n\}$ is a discrete martingale, then for any $\lambda>0$,
	\begin{equation}\label{x19}
		\mu\{ f_n -f_0>\lambda\}\le \exp\left(-\lambda^2/2\|\ZS(f)\|_\infty^2\right).
	\end{equation} 
\end{OldTheorem}
In fact, this result was stated in \cite{CWW} for the dyadic filtration in the unit cube $[0,1)^n$, but the proof works for an arbitrary discrete filtration (see \cite{CWW}, Theorem 3.1). The proof of \e{x19} stated in \cite{CWW} was suggested by Herman Rubin, which replaces a much longer argument of the authors of \cite{CWW}. Namely, the proof is adapted to the proof of classical Azuma-Hoeffding's martingale inequality \cite{Azu, Hoe}. It is based on a fundamental identity of sequential analysis known in statistics and Hoeffding's lemma \cite{Hoe}. Ivanishvili and Trail in \cite{IvTr} proved a version of inequality \e{x19} for the classical martingale square function \e{x53}, considering homogeneous filtrations $\{\ZF_k\}$, that is  if $U\in \ZD_{k}$ is the children of an element $V\in \ZD_{k-1}$, then $\mu(U)\ge \alpha\mu(V)$, where $0<\alpha\le 1/2$ is a common constant for all such choices. 
\begin{OldTheorem}[\cite{IvTr}]\label{OT2}
	Let $f=\{f_0,f_1,\ldots,f_n\}$ be a discrete martingale with respect to an $\alpha$-homogeneous filtration.  Then
	\begin{equation}\label{x20}
		\mu\{x\in X:\, f_n-f_0>\lambda\}\le \exp\left(-\alpha \lambda^2/\|S(f)\|_\infty^2\right),\quad \lambda>0.
	\end{equation} 
\end{OldTheorem}
A counterexample provided in \cite{IvTr} shows that a sub-Gaussian estimate like \e{x20}, involving the classical square function, fails for non-homogeneous filtrations.

\subsection{Khintchine type inequalities}
One can state Theorems \ref{OT1} and \ref{OT2} in the terms of a martingale difference sequences 
\begin{equation}\label{x56}
	d=\{d_1,d_2,\ldots,d_n\}, 
\end{equation}
substituting $f_n-f_0=\sum_{k=1}^nd_k$ in inequalities \e{x19} and \e{x20}. For the sake of convenience in the sequel everything will be stated in terms  of \e{x56}, denoting the square function \e{x42} by $\ZS(d)$. Denote
\begin{equation}\label{x30}
	A_{p,n}=\sup_{d=\{d_1,d_2,\ldots,d_n\}}\frac{	\left\|\sum_{k=1}^nd_k\right\|_p}{\|\ZS(d)\|_\infty},
\end{equation}
where $\sup$ is taken over all non-trivial martingale-difference sequences \e{x56} with a fixed number of elements $n$.  Obviously we have $A_{p,n}\le A_{p,n+1}$. Taking $\sup$ in \e{x30} over all martingale-difference sequences \e{x56}, without restriction on the number of elements, we get another constant $A_p$, for which clearly we have $A_p=\lim_{n\to\infty}A_{p,n}$. Considering only Haar or Rademacher martingales in \e{x30}, the corresponding constants will be denoted by $A_{p,n}({\rm Haar})$ and $A_{p,n}({\rm Rademacher})$ respectively. Obviously,
\begin{equation}\label{x45}
	A_{p,n}\ge A_{p,n}({\rm Haar})\ge A_{p,n}({\rm Rademacher}).
\end{equation}
It is well known that
\begin{equation}\label{x65}
	A_p({\rm Rademacher})=2^{1/2}\big(\Gamma((p+1)/2))/\sqrt{\pi}\big)^{1/p},\quad p>2,
\end{equation}
and equivalently, every Rademacher sum satisfies the inequality
\begin{equation}\label{x55}
	\left\|\sum_{k=1}^na_kr_k\right\|_p\le 2^{1/2}\big(\Gamma((p+1)/2))/\sqrt{\pi}\big)^{1/p}\left(\sum_{k=1}^na_k^2\right)^{1/2},
\end{equation}
where the constant on the right hand side is the best possible. For even integers $p\ge 4$ the proof of this inequality goes back to Khintchine's work \cite{Khi}, which sharpness later was proved by Stechkin \cite{Ste}.  Later on Young \cite{You} established \e{x55} for all real numbers $p\ge 3$, and finally in 1982 Haagerup \cite{Haa} introduced a new method, proving sharp bound \e{x55} for all parameters $p>2$. Komorowski in \cite{Kom} proved 
\begin{equation}\label{x62}
	A_{p,n}({\rm Rademacher})=\left\|\frac{1}{\sqrt{n}}\sum_{k=1}^nr_k\right\|_p,\quad p\ge 3,
\end{equation}
i.e.  if $p\ge 3$ and $n$ are fixed, then the best constant in \e{x55} is $\left\|\frac{1}{\sqrt{n}}\sum_{k=1}^nr_k\right\|_p$. 
For even $p>3$ this relation earlier was established by Efron \cite{Efr} and Eaton \cite{Eat}. The proof of \e{x62} for general parameters $p\ge 3$ given in \cite{Kom} is based on an inequality provided in \cite{Eat}. See also \cite{PeSh} for a detailed review of the subject.

\begin{remark}
	Recall that sub-Gaussian estimate \e{x19} imply a Khintchine type inequality $A_p\le c\sqrt p$, but this is a too rough approach to obtaining the exact value of $A_p$ even if the constant in \e{x19} is sharp. 
\end{remark}
\subsection{Main results}
In this paper we prove that inequalities in \e{x45} must actually be equalities whenever $p\ge 3$. This implies extensions of Khintchine's inequality \e{x55} and relation \e{x62} for general discrete martingales. The main results of the paper are the following theorems.
\begin{theorem}\label{T1}
	For any $p> 2$ we have  $	A_{p,n}=A_{p,n}({\rm Haar})$.
\end{theorem}
\begin{theorem}\label{T2}
	If $p\ge 3$, then 
	\begin{equation}\label{x10}
	A_{p,n}=A_{p,n}({\rm Haar})=A_{p,n}({\rm Rademacher})=\left\|\frac{1}{\sqrt{n}}\sum_{k=1}^nr_k\right\|_p,
\end{equation}
where $r_k$ are the Rademacher functions.
\end{theorem}
\trm{T2} is an extension of the results of \cite{Kom} (see \e{x62}) for general martingale-differences. Relation \e{x10} implies the following Khintchine type inequality.
\begin{corollary}\label{C1}
	If $p\ge 3$, then for any discrete martingale-difference \e{x56} we have the bound
	\begin{equation}\label{x31}
	\left\|\sum_{k=1}^nd_k\right\|_p\le \left\|\frac{1}{\sqrt{n}}\sum_{k=1}^nr_k\right\|_p \|\ZS(d)\|_\infty,
	\end{equation}
which is sharp, since for $d_k=r_k/\sqrt{n}$ we have equality in \e{x31}.
\end{corollary}
It follows from the central limit theorem that
\begin{equation}
	A_{p,n}=\left\|\frac{1}{\sqrt{n}}\sum_{k=1}^nr_k\right\|_p \to 2^{1/2}\big(\Gamma((p+1)/2))/\sqrt{\pi}\big)^{1/p} \hbox{ as } n\to\infty
\end{equation}
and applying the sharpness of \e{x31}, we can say that $A_{p,n}$ is increasing with respect to $n$ if $p\ge 3$. Thus from \e{x31} it follows that 
\begin{equation}\label{x58}
	\left\|\sum_{k=1}^nd_k\right\|_p\le 2^{1/2}\big(\Gamma((p+1)/2))/\sqrt{\pi}\big)^{1/p}\|\ZS(d)\|_\infty,\quad p\ge 3,
\end{equation}
which is an extension of inequality \e{x55} in the case $p\ge 3$ for general discrete martingale-difference sequences. Also we note that in the case of dyadic martingale the square function $\ZS(d)$ in \e{x58} can be replaced by the classical one $S(d)$. In particular, 
\begin{corollary}\label{C2}
If $\{h_k\}$ is the Haar system defined by \e{x66} and $p\ge 3$, then the inequality 
\begin{equation}\label{x80}
	\left\|\sum_{k=1}^na_kh_k\right\|_p\le 2^{1/2}\big(\Gamma((p+1)/2))/\sqrt{\pi}\big)^{1/p}\left\|\left(\sum_{k=1}^na_k^2h_k^2\right)^{1/2}\right\|_\infty, \quad p\ge 3,
\end{equation}
holds for any coefficients $a_k$.
\end{corollary}
Let us also state the following sharp sub-Gaussian inequalities.
\begin{corollary}\label{C3}
	For any discrete martingale-difference \e{x56} and for any number $0<\lambda <1/2$ we have
	\begin{equation}\label{x59}
		\ZE\left[\exp\left(\lambda\cdot \left(\frac{\sum_{k=1}^nd_k}{\|\ZS(d)\|_\infty}\right)^2\right)\right]\le \frac{1}{\sqrt{1-2\lambda}},
	\end{equation} 
where the constant on the right hand side is the best possible. Moreover, it is the best constant if we consider only dyadic or Rademacher martingales in \e{x59} . 
\end{corollary}
Consider the Orlicz space $L^\psi$ of random variables, corresponding to the Young function $\psi(t)=e^{t^2}-1$, and equipped with the Luxembourg norm 
\begin{equation}
	\|f\|_\psi=\inf\{u>0:\, \ZE\left[\psi\left(f/u\right)\right]\le 1\}.
\end{equation}
Applying \e{x59} we can immediately get the following sharp bound.
\begin{corollary}\label{C4}
	If $\psi(t)=e^{x^2}-1$, then for any discrete martingale-difference \e{x56} it holds the sharp bound
	\begin{equation}\label{x60}
		\left\|\sum_{k=1}^nd_k\right\|_\psi \le \sqrt{\frac{8}{3}}\cdot \|\ZS(d)\|_\infty.
	\end{equation} 
\end{corollary}
\begin{remark}
	It was proved by Peskir \cite{Pes} the inequality
\begin{equation}\label{x63}
	\left\|\sum_{k=1}^nX_k\right\|_\psi \le \sqrt{\frac{8}{3}}\cdot \left(\sum_{k=1}^n\|X_k\|_\infty^2\right)^{1/2}
\end{equation}
for any sequence $X=\{X_k:\, k=1,2,\ldots,n\}$ of independent symmetric random variables. Inequality \e{x63} can be easily deduced from \e{x60}. Indeed, if $X$ is a discrete sequence, then it becomes a martingale-difference on a discrete filtration. So \e{x60} holds, since using independence we will have
\begin{equation}\label{x61}
	\ZS(X)=\left(\sum_{k=1}^n\|X_k\|_\infty^2\right)^{1/2}.
\end{equation}
The general case of \e{x63} may be reduced to the discrete case, applying a standard approximation argument.  
We refer also related papers  \cite{Whi, Ros}, where authors prove Khintchine type inequality for independent symmetric random variables, with other square functions in the spirit of \e{x63}.
\end{remark}
\begin{remark}
	Theorems \ref{T1} and \ref{T2} show certain extreme properties of Haar and Radema\-cher random variables in a class of martingale difference sequences, implying extensions of some properties of Rademacher system to general martingales. Some extreme properties of Rademacher system in a class of uniformly bounded martingale differences were considered in \cite{Kar4, Kar1, Kwa} (see also \cite{Ast}, chap. 9). Moreover, those papers consider more general sequences, namely multiplicative sequences of bounded random variables $\phi_n$, $n=1,2,\ldots$, i.e.  ${\textbf E}\left[\prod_{j\in A}\phi_{j}\right]=0$ for all nonempty finite subsets $A\subset \ZN$ of positive integers. In particular, the result of \cite{Kar4} states that if $G:{\mathbb R}^n\to {\mathbb R}^+$ is a function convex with respect to each variables and $\phi=\{\phi_k:\, k=1,2,\ldots, n\}$ is a  multiplicative system of random variables, satisfying  $A_k\le \phi_k(x)\le B_k$, then 
	\begin{equation}\label{a33}
		\textbf{E}\left[G\left(\phi_1,\ldots,\phi_n\right)\right]\le \textbf{E}\left[G\left(\xi_1,\ldots,\xi_n\right)\right],
	\end{equation}
	where $\xi_k$ are the $\{A_k,B_k\}$- valued independent mean zero random variables. The case of $A_k=-1$ and $B_k=1$ of inequality \e{a33} was proved in \cite{Kwa}.
\end{remark}
\begin{remark}
	It is a remarkable result of Wang \cite{Wang} that for any $p\ge 3$ and any conditionally symmetric martingale-difference $d=\{d_1,d_2,\ldots,d_n\}$ we have the sharp bound
	\begin{equation}
	\left\|\sum_{k=1}^nd_k\right\|_p\le R_p\|S(d)\|_p,\quad p\ge 3
	\end{equation}
	where $R_p$ is the rightmost zero of the Hermite function $H_p(x)$. Here $H_p(x)$ is the solution of
	the Hermite differential equation
	\begin{equation}
		H''_p-xH'_p+pH_p=0.
	\end{equation}
It is known that $R_p=O(\sqrt{p})$.
\end{remark}
\section{Proof of \trm{T1}}
Since a sequence of random variables may form a martingale difference with respect to different filtrations, the square function \e{x42} 
strongly depends on both random variables $d_k$ and the partitions (or filtration) $\ZD_k$. So if a sequence of functions  $d_1,\ldots,d_n$ form a martingale difference relative to an increasing sequence of partitions $\ZD_1,\ldots,\ZD_n$, then we say the collection 
	\begin{equation}\label{x41}
	d=\{d_1,\ldots,d_n|\ZD_1,\ldots,\ZD_n\},\quad 1\le k<n, 
\end{equation}
forms a MD-system (martingale difference system).  Denote by $\pr(A)$ the parent of a given element $A\in \ZD_j$, $j>1$, which is the unique element of $\ZD_{j-1}$, containing $A$. Set
\begin{equation}
	\|d\|_p=\left\|\sum_{j=1}^nd_k\right\|_p,\quad  \ZU(d)=\|d\|_p/\|\ZS(d)\|_\infty, \quad p\ge 1,
\end{equation}
where $\ZS(d)$ is the square function \e{x42}. Hence for the constant $A_{p,n}$ we have
\begin{equation}\label{x43}
	A_{p,n}=\sup_d\ZU(d),
\end{equation}
where $\sup$ is taken over all non-trivial MD-systems \e{x41}. One can take the $\sup$ in \e{x43} over MD-systems, satisfying certain property (P). In that case we will use the notation $A_{p,n}(P)$. We say that a MD-system \e{x41} is $k$-dyadic with $2\le k\le n$, if each $V\in \ZD_j$ with $1\le j<k$ has exactly two children in $\ZD_{j+1}$, and 1-dyadic if $d_1$ takes two values. The property $n$-dyadic is the same as dyadic defined in the introduction. Observe that to prove \trm{T1} it is enough to prove 
\begin{equation}\label{x44}
	A_{p,n}((k-1)-{\rm dyadic})=A_{p,n}(k-{\rm dyadic})\hbox{ for all }1< k\le n. 
\end{equation}
Introduce an intermediate property ({\rm IP}) for MD-system \e{x41} that requires

\vspace{3mm}
{\rm IP}) $(k-1)$-dyadic \& $|d_k(x)|=\|\ZI_V\cdot d_k\|_\infty,\quad x\in V\in \ZB_{k-1}$.
\vspace{3mm}

Hence, equality \e{x44} and so the theorem will be proved if we prove following two relations:

\begin{enumerate}
	\item [R1:]$A_{p,n}((k-1)-{\rm dyadic})=A_{p,n}({\rm IP})$,
	\item [R2:]$A_{p,n}({\rm IP})=A_{p,n}(k-{\rm dyadic})$.
\end{enumerate}
\begin{proof}[Proof of R1]
Without loss of generality we can suppose that $X=[0,1)$ and each partition $\ZD_k$ consists of intervals of the form $[a,b)$.  It is enough to prove that for any $(k-1)$-dyadic MD-system \e{x41} there exists a MD-system 
$\bar d$ with the {\rm IP}-property such that $\ZU(d)\le \ZU(\bar d)$. Let \e{x41} be a $(k-1)$-dyadic MD-system.  The desired MD-system will be obtained from $d$ by reconstructing only the $k$'th function $d_{k}$ and the partitions $\ZD_{k},\ldots, \ZD_n$. Namely, after these changes we will have a MD-system
\begin{equation}\label{x35}
		\bar d=\{d_1,\ldots,d_{k-1},\bar d_k,d_{k+1},\ldots,  d_n|\ZD_1,\ldots,\ZD_{k-1},\bar \ZD_k,\ldots, \bar \ZD_n\},
\end{equation}
which in any cases will be $(k-1)$-dyadic. The function $d_k$ is constant on each interval $J\in \ZD_k$, so we denote this constant by $\xi(J)$.
For $I=[a,b)\in \ZD_n$ consider the unique interval $\bar I\in \ZD_k$ such that $\bar I\supset I$. We have, $d_k(x)=\xi(\bar I)$ whenever $x\in \bar I$, and $|\xi(\bar I)|\le \|\ZI_{\pr(\bar I)}\cdot d_k\|_\infty$. Divide $I$ into two intervals $I^+$ and $I^-$, setting
	\begin{align}
		&\lambda=\lambda(\bar I)=\frac{1}{2}\left(1+\frac{\xi(\bar I)}{ \|\ZI_{\pr(\bar I)}\cdot d_k\|_\infty}\right)\in [0,1],\label{x13}\\
		&c=(1-\lambda) a+\lambda b\in [a,b],\label{x14}\\
		&I^+=[a,c),\quad I^-=[c,b).\label{x38}
	\end{align}
Using these notations, denote
	\begin{equation}\label{x34}
		\bar d_k(x)=\sum_{I\in\ZD_n}\|\ZI_{\pr(\bar I)}\cdot d_k\|_\infty\left(\ZI_{I^+}(x)-\ZI_{I^-}(x)\right).
	\end{equation}
The required properties of $\bar d_k$ are the followings:
\begin{align}
	& |\bar d_k(x)|=\|\ZI_V\cdot d_k\|_\infty\hbox{ whenever }x\in V\in \ZB_{k-1},\label{x67}\\
	&\int_I\bar d_k=\int_Id_k \hbox{ for any }I\in \ZD_n,\label{x16}
\end{align}
which easily follow from \e{x38} and \e{x34}.
Then define the partition $\bar \ZD_j$, $k\le j\le n$,  replacing each interval $V\in \ZD_j$ by two sets
	\begin{equation}\label{x26}
		V^+=\bigcup_{I\in \ZD_n,\,I\subset V}I^+,\quad V^-=\bigcup_{I\in \ZD_n,\,I\subset V}I^-.
	\end{equation}
Hence the construction of the elements of the MD-system \e{x35} is complete. Clearly, the new MD-system $\bar d$ is still $(k-1)$-dyadic. Thus, taking into account \e{x67}, we can say that $\bar d$ satisfies {\rm IP}-property. Let us prove that \e{x35} is a MD-system, i.e. the functions in \e{x35} form a martingale difference with respect to the partitions of \e{x35}. We need to check the martingale difference property only for the elements $\bar d_k$ and $d_j$ with $j>k$. 
Applying \e{x16}, for any $V\in \ZD_{k-1}$ we obtain
	\begin{equation}
		\int_V\bar d_k=\sum_{ I\in \ZD_n,\,I\subset V}\int_I \bar d_k=\sum_{I\in \ZD_n,\,I\subset V}\int_I d_k=\int_Vd_k=0
	\end{equation}
that means $\ZE(\bar d_k|\ZD_{k-1})=0$. To show $\ZE(d_j|\bar \ZD_{j-1})=0$ for $j>k$ choose an element $V^+\in \bar \ZD_{j-1}$ generated by a $V\in \ZD_{j-1}$ (see \e{x26}). If $I\in \ZD_n$ and $I\subset V$, then we have $V\subseteq \bar I\in \ZD_k$. We have $V\subseteq J$ for some $J\in \ZD_k$. By \e{x13} and \e{x14},  we have $|I^+|=\lambda |I|$ with a common parameter $\lambda=\lambda(\bar I)$. Thus, since the functions $d_j$ are constant on each interval $I\in \ZD_n$, using \e{x26}, for $j>k$ we can write
	\begin{equation*}
		\int_{V^+} d_j=\sum_{I\in \ZD_n,\,I\subset V} \int_{I^+}d_j=\lambda\sum_{I\in \ZD_n,\, I\subset V} \int_{I}d_j=\lambda \int_Vd_j=0,
	\end{equation*}
where the latter follows from the martingale-difference property of the initial MD-system $d$. Similarly, $\int_{V^-} d_j=0$. This implies $\ZE(d_j|\bar \ZD_{j-1})=0$ and so \e{x35} is a MD-system. By \e{x16} we can write
\begin{equation}
	\int_I \left(\bar d_k+\sum_{j\ne k}^nd_j\right)= \int_I\sum_{j=1}^nd_j,\quad I\in \ZD_n.
\end{equation}
Thus, since the functions $d_j$ are constant on $I$, applying Jessen's inequality, we get the bound
	\begin{equation}\label{a18}
		\|\bar d\|_p^p=\sum_{I\in \ZD_n}\int_I\left|\bar d_k+\sum_{j\neq k} d_j\right|^p\ge\sum_{I\in \ZD_n} \int_I\left|\sum_{j=1}^n d_j\right|^p=\|d\|_p^p,
	\end{equation}
By the definition of the square function, both $\ZS \bar d(x)$ and $\ZS d(x)$ are square roots of $n$-term sums. We will show that the corresponding terms in these sums coincide. The equality of first $k-1$ terms in these sums is immediate, since the first $k-1$ elements of MD-systems $d$ and $\bar d$ coincide. For the  terms with indexes $j\ge k$ we have the following. For any $x\in [0,1)$ there is a unique sequence $V_j\in \ZD_j$, $j=1,2,\ldots,n$ such that $x\in V_j$. For an appropriate choice of signs $\varepsilon_j=\pm $ we will also have $x\in \cap_{k\le j\le n}V_j^{\varepsilon_j}$. Applying the definition of $\bar \ZD_j$, $j\ge k$, one can check,
	\begin{align}
		&\|\ZI_{V_{j-1}}d_j\|_\infty=\|\ZI_{V_{j-1}^{\varepsilon_j}}d_j\|_\infty,\quad j>k,\\
		&\|\ZI_{V_{k-1}}\bar d_k\|_\infty=\|\ZI_{V_{k-1}} d_k\|_\infty (\hbox{see }\e{x67}).
	\end{align}
	Thus we obtain equality of all the terms in the sums of $\ZS \bar d(x)$ and $\ZS d(x)$. So we get $\ZS \bar d(x)=\ZS d(x)$ everywhere. Combining this with \e{a18}, we obtain $\ZU(d)\le \ZU(\bar d)$, completing the construction of \e{x35} and the proof of relation R1.
\end{proof}
\begin{proof}[Proof of R2]
	Now suppose that \e{x41} satisfies the intermediate-property {\rm IP}. We need to get a $k$-dyadic MD-system $\bar d$ such that $\ZU(d)\le \ZU(\bar d)$. The only lack in the properties of $d$ is that some sets in $\ZD_{k-1}$ may have more that two children. Starting our construction, we will replace $d_j$ by $\bar d_j$ for $j>k$ and $\ZD_j$ by $\bar \ZD_j$ for $j\ge k$. Thus our new MD-system will look like
	\begin{equation}\label{x39}
		\bar {d}=\{d_1,\ldots,d_{k-1},d_k,\bar d_{k+1}\ldots,  \bar d_n|\ZD_1,\ldots,\ZD_{k-1},\bar {\ZD}_k,\ldots, \bar{ \ZD}_n\}.
	\end{equation}
For any $V\in \ZD_{k-1}$ consider two sets
\begin{equation}
	V^\pm=\{x\in X:\,d_k(x)=\pm \|\ZI_V\cdot d_k\|_\infty\},
\end{equation}
which we define as child sets of the element $V$ in $\bar {\ZD}_k$. Hence,
\begin{equation}
	\bar \ZD_k=\{V^\pm:\, V\in \ZD_{k-1}\}.
\end{equation}
If we stop our construction here the sets of $\ZD_{k-1}$ would exactly have two children in $\bar {\ZD}_k$ and we will have $\|d\|_p=\|\bar d\|_p$ (that is good), but the square function of a new MD-system $\bar d$ can be bigger than that of $d$ (which is bad). So we have to continue the reconstruction to fix the problem. Suppose for a fixed $V\in \ZD_{k-1}$ the integrals
	\begin{align}\label{x36}
		&\frac{1}{|J|}\int_{J}\left|\sum_{j=1}^n d_j\right|^p,\quad 	J\in \ZD_k,\,  J\subset V^+,\\
		&\frac{1}{|J|}\int_{J}\left|\sum_{j=1}^n d_j\right|^p,\quad J\in \ZD_k,\,  J\subset V^-,
	\end{align}
attains their maximum for the intervals $J=J_V^+$ and $J=J_V^-$ respectively.  Let $\Lambda_{I\to J}$ denote the dilation from $[0,1)$ to $[0,1)$, which is a linear mapping from an interval $I\subset [0,1)$ onto an interval $J\subset [0,1)$. We define the elements $\bar{d}_j(x)$ and $\bar{\ZD}_j$ for $j>k$ on the intervals $J\in \ZD_k$, $J\subset V^\pm$ making dilation of those elements from $J_V^\pm$ onto each such $J$. Namely,
\begin{align}
	&\bar{d}_j(x)=\sum_{V\in \ZB_{k-1}}\bigg(\sum_{J\in \ZD_k:\,  J\subset V^+}d_j(\Lambda_{J\to J_V^+}(x))\cdot \ZI_{J}(x)\\
	&\qquad\qquad\qquad\qquad+\sum_{J\in \ZD_k:\,  J\subset V^-}d_j(\Lambda_{J\to J_V^-}(x))\cdot \ZI_{J}(x)\bigg),\quad  j>k,\\
	&\bar{\ZD}_j^+=\bigg\{\bigcup_{J\in \ZD_k:\,  J\subset V^+}\Lambda_{J\to J_V^+}^{-1}(I):\,V\in \ZD_{k-1},\,I\in \ZD_j,\, I\subset J_V^+\bigg\},\\
	&\bar{\ZD}_j^-=\bigg\{\bigcup_{J\in \ZD_k:\,  J\subset V^-}\Lambda_{J\to J_V^-}^{-1}(I):\,V\in \ZD_{k-1},\,I\in \ZD_j,\, I\subset J_V^-\bigg\},\\
	&\bar{\ZD}_j=\bar{\ZD}_j^+\cup \bar{\ZD}_j^-,\quad j>k.
\end{align}
 Hence the elements of \e{x39} have been already defined. One can check that 
$\bar{d}$ is $k$-dyadic MD-system and satisfies $\ZS(\bar d)\le \ZS(d)$. Using the maximal property of the intervals $J_V^\pm$, we also have $\|d\|_p\le \|\bar d\|_p$. Thus, $\ZU(d)\le\ZU(\bar d)$. This completes the proof of R2-relation and so the proof of \trm{T1}.

\end{proof}

\section{Proof of \trm{T2}}
Lemmas \ref{L3} and \ref{L4} below were proved in \cite{Kom} and \cite{Whi} respectively. For the completeness we will sate the proofs of those lemmas. The proof of \lem{L3} is taken from \cite{Kom}. For \lem{L4} an alternative proof is given.  
\begin{lemma}[\cite{Kom}]\label{L3}
	For a fixed $\xi>0$ and $p>3$ the function 
	\begin{equation}\label{x28}
		u(t)=\frac{|t+\xi|^{p-1}\cdot \sign(t+\xi)+|t-\xi|^{p-1}\cdot \sign(t-\xi)}{t}
	\end{equation}
is increasing over $t>0$. 
\end{lemma}
\begin{proof}
For the derivative of function \e{x28} we have 
\begin{equation}\label{x81}
	u'(t)=\frac{|t+\xi|^{p-2}((p-2)t-\xi)+|t-\xi|^{p-2}((p-2)t+\xi)}{t^2}.
\end{equation}
The derivative of the numerator of \e{x81} is equal
\begin{equation*}
	t(p-1)(p-2)\left(|t+\xi|^{p-3}\sign(t+\xi)+|t-\xi|^{p-3}\sign(t-\xi)\right),
\end{equation*}
which is positive if $t>0$. Since this numerator is zero at $t=0$ and its derivative is greater than zero we get $u'(t)\ge 0$.
\end{proof}

\begin{lemma}[\cite{Whi}]\label{L4}
	Let $p>3$ and $\xi\in \ZR$ be a fixed real number. Then the sum
	\begin{equation}\label{x1}
		|x+y+\xi|^p+|x+y-\xi|^p+|y-x+\xi|^p+|y-x-\xi|^p
	\end{equation}
	where $x,y$ satisfy $x^2+y^2=r^2$, attains its maximum if and only if $|x|=|y|=r/\sqrt{2}$.
\end{lemma}
\begin{proof}

	Without loss of generality we can suppose that $\xi\ge 0$, $r=1$ and $y\ge 1/\sqrt{2}\ge x>0$. Let $f(x)$ denote the function \e{x1}  after substitution  $y= \sqrt{1-x^2}$. It is enough to prove that $f(x)$ is increasing on the interval $(0,1/\sqrt{2})$. Applying \lem{L4} and 
	\begin{equation*}
		(|x|^p)'=p|x|^{p-1}\sign x, \quad p>1,
	\end{equation*}
	for the derivative of $f(x)$ at a point $x\in (0,1/\sqrt 2)$ we get
	\begin{align}
		f'(x)&=p\left(1-\frac{x}{y}\right)\bigg(|x+y+\xi|^{p-1}\sign(x+y+\xi)\\
		&\qquad\qquad+|x+y-\xi |^{p-1}\sign(x+y-\xi)\bigg)\\
		&+p\left(1+\frac{x}{y}\right)\bigg(|y-x+\xi|^{p-1}\sign(y-x+\xi)\\
		&\qquad\qquad+|y-x-\xi |^{p-1}\sign(y-x-\xi)\bigg)\\
		&=\frac{p(y-x)(y+x)}{y}(u(y+x)-u(y-x))>0,
	\end{align}
	where $u(t)$ is the function in \e{x28}. Thus, applying \lem{L3}, we get $f(x)$ is increasing over $(0,1/\sqrt 2)$.
\end{proof}
Applying Lemmas \ref{L4} consecutively, we get 
\begin{lemma}\label{L6}
	If $p>3$ and $\xi\in \ZR$ are fixed, then the value
	\begin{equation}
		\left\|\xi+\sum_{k=1}^na_kr_k\right\|_p
	\end{equation}
	where the Rademacher coefficients $a_k$ satisfy $\sum_{k=1}^na_k^2=r^2$, attains its maximum if and only if $|a_k|=r/\sqrt{n}$.
\end{lemma}
\begin{lemma}\label{L8}
	Let $p\ge 1$and $\xi\in \ZR$ be fixed. Then for any non-trivial function $g\in L^1(0,1)$, satisfying $\ZE(g)=0$, the function
	\begin{equation}
		h(x)=\int_0^1\left|\xi+x\cdot g(t)\right|^pdt
	\end{equation}
is increasing over $(0,\infty)$.
\end{lemma}
\begin{proof}
Without loss of generality it is enough to prove that 
\begin{equation}\label{x52}
	\int_0^1\left|\xi+g(t)\right|^pdt<\int_0^1\left|\xi+r\cdot g(t)\right|^pdt\hbox{ for }r>1.
\end{equation}
Applying an approximation, we can suppose that $g$ is a step function on $[0,1)$ with equal constancy intervals. Let 
\begin{equation}
	E^+=\{x\in [0,1):\, g(x)\ge 0\},\quad E^-=\{x\in [0,1):\, g(x)<0\}.
\end{equation}
We will construct a sequence of step functions $g=g_1, g_2,\ldots, g_m=r\cdot g$, with the same constancy intervals as $g$ has, such that  
\begin{align}
	&\ZE(g_k)=0,\label{x50}\\
	&0\le g_k\cdot \ZI_{E^+}\le g_{k+1}\cdot \ZI_{E^+}\le rg(x)\cdot \ZI_{E^+},\label{x51}\\ 
	&0\ge  g_k\cdot \ZI_{E^-}\ge g_{k+1}\cdot \ZI_{E^-}\ge rg(x)\cdot \ZI_{E^-},\label{x48}\\
	&\int_X\left|\xi+g_k(t)\right|^pdt< \int_X\left|\xi+g_{k+1}(t)\right|^pdt,\quad k=1,2,\ldots, m-1.\label{x49}
\end{align}
Suppose by induction we have already defined the functions $g_1,\ldots, g_l$ such that the above conditions hold for $k=1,\ldots, l-1$. If $g_l=rg$, then the process will stop. Otherwise, applying \e{x50}, one can say that $g_l$ has at least two constancy intervals $I^+\subset E^+$ and $I^-\subset E^-$ such that
\begin{align}
	&rg(x)=b_+>g_l(x)=a_+\ge 0,\, x\in I^+,\\
	 &rg(x)=b_-<g_l(x)=a_-<0,\, x\in I^-.
\end{align}
Denote 
\begin{equation}
	\lambda=\min\{b_+- a_+, a_--b_-\}>0.
\end{equation}
We define $g_{l+1}$ by changing the values of $g_l$, on the intervals $I^+$ and $I^-$ by $a_++\lambda$ and $a_--\lambda$ respectively. Clearly, the conditions \e{x50}, \e{x51} and \e{x48} are satisfied for $k=l$ too. To show \e{x49} it is enough to observe the inequality
\begin{equation*}
	\int_{I^+\cup I^-}\left|\xi+g_l(t)\right|^pdt\le \int_{I^+\cup I^-}\left|\xi+g_{l+1}(t)\right|^pdt,
\end{equation*}
which is the same as the numerical inequality
\begin{equation*}
	|\xi+a_+|^p+|\xi+a_-|^p<|\xi+a_++\lambda|^p+|\xi+a_--\lambda|^p.
\end{equation*}
The latter follows from the fact that the function $t(x)=|c+x|^p+|c-x|^p$ is increasing when $x>0$. After this step of induction we will get one more interval (either $I^+$ or $I^-$), where $g_{l+1}$ coincides with $rg$. Thus, continuing the induction we will finally get a function $g_m=rg$. Thus, we will have \e{x52}, completing the proof of lemma.
\end{proof}
\begin{proof}[Proof of \trm{T2}]
 Recall that $\ZD_0=\{X\}$. We say a dyadic MD-system \e{x41} is $m$-Rademacher, $1\le m\le n$,  if for any $V\in \ZD_{m-1}$ we have
\begin{equation}\label{x68}
	|d_{m}(x)|=|d_{m+1}(x)|=\ldots=|d_n(x)|=c,\quad x\in V. 
\end{equation}
Here $n$-Rademacher is nothing but to be dyadic, while $1$-Rademacher property requires 
\begin{equation}
|d_1(x)|=|d_2(x)|=\ldots =|d_n(x)|=c,\quad x\in X.
\end{equation}
We will first prove that 
\begin{equation}\label{x46}
	A_{p,n}(m-{\rm Rademacher})=A_{p,n}((m-1)-{\rm Rademacher}), \quad 1<m\le n.
\end{equation}
To this end it is enough to prove that for any $m$-Rademacher system $d$ there exists a $(m-1)$-Rademacher system $\bar d$ such that $\ZU(d)\le \ZU(\bar d)$.  
We apply two procedures inside arbitrary $V\in \ZD_{m-2}$ , changing the values of some functions $d_k$ on $V$.  

{\it Procedure 1: }Since our MD-system is dyadic, $V\in \ZD_{m-2}$ has exactly two children $V^+,V^-\in \ZD_{m-1}$ with $\mu(V^+)=\mu(V^-)=\mu(V)/2$.
According to the property being $m$-Rademacher we have
\begin{align}
	&|d_{m}(x)|=|d_{m+1}(x)|=\ldots=|d_n(x)|=c_+,\quad x\in V^+,\\
	&|d_{m}(x)|=|d_{m+1}(x)|=\ldots=|d_n(x)|=c_-,\quad x\in V^-,\\
	&|d_k(x)|=c_k,\quad x\in V,\quad 1\le k<m.
\end{align}
We have
\begin{align}
	&(\ZS d(x))^2=\sum_{j=1}^{m-1}c_k^2+(n-m+1)(c_+)^2,\quad x\in V^+,\\
	&(\ZS d(x))^2=\sum_{j=1}^{m-1}c_k^2+(n-m+1)(c_-)^2,\quad x\in V^-.
\end{align}
Without loss of generality we can suppose that $c_-\le c_+$. In this case we change the values of functions $d_m,d_{m+1},\ldots,d_n$ on the set $V^-$ just multiplying those by $c_+/c_-$. One can check that after this change the $L^\infty$-norm of the square function $\ZS(d)$ will not be changed, but $\|d\|_p$ will increase because of \lem{L8}. Therefore $\ZU(d)$ increases. After these changes we will get
\begin{equation}\label{x83}
	|d_m(x)|=\ldots= |d_n(x)|=c_V \hbox{ for every }x\in V,
\end{equation}
where $c_V$ is either $c_+$ or $c_-$.

{\it Procedure 2: } So after Procedure-1 we will have \e{x83} for any $V\in \ZD_{m-2}$. If along with \e{x83} we also had  
\begin{equation}\label{x69}
	|d_{m-1}(x)|=c_V
\end{equation}
with the same constant, then we could stop our procedure. Perhaps \e{x69} may fail. One can define functions $\bar d_{m-1}, \bar d_m,\ldots, \bar d_n$ such that the $\{d_1,\ldots, d_{n-2},\bar d_{m-1}, \bar d_m,\ldots, \bar d_n\}$ forms a dyadic MD-system with respect to the same partitions $\ZD_j$ and on each $V\in \ZD_{m-2}$ we have
\begin{align}
&|\bar d_{m-1}(x)|= |\bar d_m(x)|=\ldots= |\bar d_n(x)|=\constant_V,\quad x\in V,\\
&\sum_{k=m-1}^n|\bar d_k(x)|^2=\sum_{k=m-1}^n| d_k(x)|^2,\quad x\in V.
\end{align}
Clearly, we will have $\|\ZS(\bar d)\|_\infty= \|\ZS(d)\|_\infty$ . Applying \lem{L6}, we also get $\|d\|_p\le \|\bar d\|_p$ and therefore $\ZU(d)\le \ZU(\bar d)$. Completing Procedure-2, we will have proved required relation \e{x46}. 
To finalize the proof of theorem it remains to apply \e{x46} consecutively. Then we get 
\begin{equation}
	A_{p,n}({\rm diadic})=A_{p,n}(1-{\rm Rademacher})=\left\|\frac{1}{\sqrt n}\sum_{k=1}^nr_k\right\|_p,
\end{equation}
where $r_k$ are the Rademacher functions.
\end{proof}
\section{Proofs of corollaries}
Corollaries \ref{C1} and \ref{C2} immediately follow from Theorems \ref{T1} and \ref{T2}.
\begin{proof}[Proof of \cor{C3}]
Writing Taylor series of function $e^{t}$, we get
\begin{equation}\label{x70}
	\ZE\left[\exp\left(\lambda\cdot \left(\frac{\sum_{j=1}^nd_j}{\|\ZS(d)\|_\infty}\right)^2\right)\right]=
	\sum_{k=0}^\infty\frac{\lambda^k}{ k!}\left\|\frac{\sum_{j=1}^nd_j}{\|\ZS(d)\|_\infty}\right\|_{2k}^{2k}
\end{equation}
Then applying \trm{T2}, we can say that the quantities $\left\|\sum_{j=1}^nd_j/\|\ZS(d)\|_\infty\right\|_{2k}^{2k}$ attain their maximums when $d_j=r_j/\sqrt{n}$. On the other hand we have
\begin{equation}
	\left\|\frac{1}{\sqrt{n}}\sum_{j=1}^nr_j\right\|_{2k}^{2k}<2^{k}\Gamma(k+1/2)/\sqrt{\pi}=\frac{(2k)!}{2^k \cdot k!}.
\end{equation} 
Thus, using the Binomial series of $(1-x)^{-1/2}$, \e{x70} can be estimated by
\begin{equation*}
	\sum_{k=0}^\infty \frac{(2k)!}{2^k(k!)^2}\lambda^k=\frac{1}{\sqrt{1-2\lambda}}.
\end{equation*}
\end{proof}
\begin{proof}[Proof of \cor{C4}]
Let $f=\sum_{j=1}^nd_j/\|\ZS(d)\|_\infty$
and $u>0$. Applying \e{x59}, we can write
\begin{equation}\label{x71}
\ZE(\psi(f/u))\le \frac{1}{\sqrt{1-2/u^2}}-1.
\end{equation}
Hence, according to the definition of Luxembourg norm we obtain $\|f\|_\psi\le \sqrt{8/3}$, since for $u=\sqrt{8/3}$ the right hand side of \e{x71} is $1$.
\end{proof}
\section{Open problems} 
The question whether inequality \e{x58} holds for parameters $2<p<3$ remains open. Note that according to \trm{T1} it is enough to consider Haar martingales insted of general discrete martingale-differences. Some standard calculations show that function \e{x31} is not increasing if $2<p<3$: it decreases over $(0,c\xi)$ and increases over $(c\xi, \infty)$, where $c\approx 1.2$ is an absolute constant. This is the only issue that do not allow to apply the same technique in the case $2<p<3$. So let us state the following.
\begin{problem}
	Prove inequality \e{x80} for the parameters $2<p<3$.
\end{problem}

In general, \e{x10} and \e{x31} fail for some parameters $2<p<3$, even for the Rademacher sums. Indeed, validity of \e{x31} for some $p>2$ implies increaseness of $\|\sum_{k=1}^nr_k/\sqrt n\|_p$ with respect to $n$, while for $p=2.5$ a standard calculation shows 
\begin{align*}
	\left\|\frac{1}{\sqrt{2}}\left(r_1+r_2\right)\right\|_p^p=\frac{2^{p/2}}{2}>\left\|\frac{1}{\sqrt{3}}\left(r_1+r_2+r_3\right)\right\|_p^p=\frac{3^{p/2}}{4}+\frac{3}{4\cdot 3^{p/2}}.
\end{align*}
 \begin{problem}
 	Find the values of $A_{p,n}({\rm Haar})$ and $A_{p,n}({\rm Rademacher})$ for $2<p<3$.
 \end{problem}
For a discussion of other problems, concerning Khintchine type inequalities we refer review paper \cite{PeSh}.

\bibliographystyle{plain}


\end{document}